\documentclass[10pt,english]{amsart}

\usepackage[T1]{fontenc}
\usepackage[latin9]{inputenc}
\usepackage{amsmath}
\usepackage{amsthm}
\usepackage{txfonts}
\usepackage[letterpaper,margin=3cm]{geometry}
\usepackage{babel}

\theoremstyle{plain}

\newtheorem{theorem}{Theorem}[section]
\newtheorem{corollary}[theorem]{Corollary}
\newtheorem{lemma}[theorem]{Lemma}

\newcommand{\RealVect}[1]{{\mathbb R}^{#1}}

\newcommand{\Rk}[3]{\frac{1}{|{#1} - {#2}|^{#3}}} 
\newcommand{\Esf}[2]{\frac{\Es\left({#1},{#2}\right)}{{#2}^{1+s/d}}} 
\newcommand{\Esfp}[2]{\frac{\Es\left({#1},{#2}\right)}{\left({#2}\right)^{1+s/d}}} 

\def\Rl{\mathbb R} 
\def\Rp{\RealVect{p}} 
\def\MA{{\mathcal M}(A)} 
\def\MAp{{\mathcal M}_1(A)}
\def\Es{{\mathcal E}_s} 
\def\Hd{{\mathcal H}^d} 
\def\glow{\underline g_{s,d}} 
\def\gup{\overline g_{s,d}} 
\def\gex{g_{s,d}} 

\DeclareMathOperator{\diam}{diam}

\DeclareMathOperator{\argmin}{arg\ min}

\begin{document}

\title{A sequence of discrete minimal energy configurations that does not converge in the weak-star topology}

\author{Matthew T. Calef}
\address{M. T. Calef:
Computational Physics (CCS-2),
Los Alamos National Laboratory,
Los Alamos, NM 87545,
USA }

\email{mcalef@lanl.gov}

\thanks{Los Alamos National Laboratory, an affirmative action/equal
  opportunity employer, is operated by Los Alamos National Security,
  LLC, for the National Nuclear Security Administration of the
  U.S. Department of Energy under contract DE-AC52-06NA25396.}

\begin{abstract}
We demonstrate a set $A$ and a value of $s$ for which the sequence of
$N$-point discrete minimal Riesz $s$-energy configurations on $A$ does
not have an asymptotic distribution in the weak-star sense as $N$
tends to infinity.
\end{abstract}

\maketitle

\section{Introduction}

In the study of classical electrostatics the representation of a
large collection of electrons on a conductor by a charge density function is
commonplace. In many cases such a representation is both physically
and mathematically grounded. Here we consider an interaction potential
derived from electrostatics and certain conductor geometries as a
means to establish limits on when such a representation of many point
charges is valid.  In particular we demonstrate that, for certain
fractal conductor geometries and Riesz potentials, there is no
limiting charge distribution.

Consider a conductor $A$ as a compact subset of $\Rp$ with
Hausdorff dimension $d$. Let $\omega_N = \{x_1,\ldots,x_N\}$ denote
the locations of $N$ electrons on $A$. The electrostatic energy of
$\omega_N$ is, up to a constant,
$$
E_s(\omega_N) := \sum_{i=1}^N \sum_{j\ne i} \Rk{x_i}{x_j}{s},
$$ where $s$ is chosen to be one. By varying $s$ the Riesz $s$-kernel
$|x-y|^{-s}$ can represent generalizations of the Coulomb
potential that decay or are singular to varying degrees. 

The infimum of the $N$-point $s$-energy is denoted by $$ \Es(A,N) :=
\inf_{\omega_N \subset A} E_s(\omega_N).
$$ We extend $\Es$ so that $\Es(A,0) = \Es(A,1) := 0$. For convenience
we exclude the trivial case that $A$ has finitely many points.  For
any positive value of $s$ the functional $E_s$ is lower
semicontinuous, therefore, by the compactness of $A$, there is at
least one configuration, which is denoted $\omega_N^s$, that satisfies
$$
\Es(A, N) = E_s(\omega_N^s),
$$ Identifying a minimal configuration $\omega_N^s$ for even simple
conductor geometries such as $\mathbb{S}^2$ is a formidable task.
Alternatively, one may study qualitative properties of $\omega_N^s$ as
$N$ tends to infinity.  A common way to do this recasts the problem in
terms of measures. For each $N$ and $s$ construct the probability
measure
$$
\mu^{s,N} = \frac 1 N \sum_{x \in \omega_N^s} \delta_x
$$ that places a scaled Dirac-measure $\delta$ at each of the points
in $\omega_N^s$.  One then considers measure-theoretic properties of
the sequence of measures $\{\mu^{s,N}\}_{N=2}^\infty$.

In the case when $s<d$ one may formulate a continuous version of
this problem as follows: Let $\MA$ denote the (unsigned) Borel
measures supported on $A$. Let $\MAp \subset\MA$ denote the Borel
probability supported on $A$. If we represent a charge density by a
measure $\mu\in\MA$ then the Riesz $s$-energy of $\mu$ is
$$
I_s(\mu) := \iint \Rk{x}{y}{s} d\mu(y) d\mu(x),
$$
and the electrostatic ($s=1$) energy of $\mu$ is $I_1(\mu)$. We may
then consider the minimization problem
$$
\mu^s = \argmin \{ I_s(\mu) \, :\, \mu \in \MAp\}.
$$ It is well known (cf~\cite{Landkof1}) that $\mu^s$ exists and is
unique, that 
$$
\lim_{N\to\infty} \frac{\Es(A,N)}{N^2} = I_s(\mu^s),
$$
and that for any function $f$ that is continuous on $A$
$$
\lim_{N\to\infty} \int f\, d\mu^{s,N} = \int f\, d\mu^s.
$$ The last condition is referred to as weak-star convergence of
$\mu^{s,N}$ to $\mu^s$. A natural interpretation is that $\mu^s$ is
the continuous equilibrium charge distribution on $A$ and that it is a
limit, as $N$ tends to infinity, of minimal $s$-energy
$N$-point configurations. These results have practical value as it is
often substantially easier to obtain finite-dimensional approximations
of $\mu^s$ than it is to determine the points that make up
$\omega_N^s$ for large $N$.

In the case when $s\ge d$ every non-zero Borel measure supported on
$A$ has infinite $s$-energy (cf.~\cite{Mattila1}) and there is no
obvious candidate for a continuous charge distribution that is the
limit of the discrete minimal energy configurations. For this range of
$s$ recent results in~\cite{HardinSaff1, BHS1, BHS2} show that when
$A$ has certain rectifiability properties, the $N$-point minimal
configurations are asymptotically uniform,
i.e. $\{\mu^{s,N}\}_{N=2}^\infty$ converges in the weak-star sense to
a suitably normalized $d$-dimensional Hausdorff measure $\Hd$
restricted to $A$. Additionally, the order of growth of $\Es(A,N)$ is
shown to be $N^{1+s/d}$ when $s>d$ and $N^2\log N$ when $s=d$.

In~\cite{CalefHardin1} a candidate for a normalized $s=d$ energy is
presented. (An alternate normalization for the case $s=d$ is presented
by Gustafsson and Putinar in~\cite{GustafssonPutinar1}. Putinar uses
this normalization to study inverse moments problems
in~\cite{Putinar1}.)  The work in~\cite{CalefHardin1, Calef2} was
motivated by the hope that one could formulate, at least in the case
$s=d$, a normalized energy that could be used to obtain results for
the limit as $N$ tends to infinity of the discrete minimal $d$-energy
configurations.  This effort was able to show that, for certain
rectifiable or fractal sets, the normalized energy is uniquely
minimized $\Hd$ restricted to $A$ and normalized to have mass $1$ --
i.e. the uniform measure -- and that $\mu^s$ converges in the
weak-star sense to this uniform measure as $s$ approaches $d$ from
below.  However, it is not yet clear what can be inferred about the
limit as $N$ grows of the discrete minimal $d$-energy configurations.

Relatedly, the results presented in~\cite{HardinSaff1} rely heavily on
characteristics -- e.g. the local structure of the configurations
$\omega^s_N$ -- that are lost in the weak-star limit. Further, results
from~\cite{BHS2} show that for $s$ sufficiently large and $A$ in a
class of self-similar fractals these local characteristics can cause
the discrete minimal energy to oscillate to leading order as $N$ tends
to infinity.  These results suggest that there are limits to what can
be obtained from any continuous minimization problem, and this paper
affirms this.  By using the oscillations in energy demonstrated
in~\cite{BHS2}, we demonstrate a value of $s$ and a set $A$ where any
sequence of minimal $N$-point $s$-energy configurations does not
converge in the weak-star sense. For such a set there can be no valid
continuous representation of the limiting distribution as the number
of points grows to infinity.

\subsection{Main Results}

We consider, for some compact $K\subset \Rp$ of Hausdorff dimension
$d$, the functions (cf.~\cite{HardinSaff1})
$$\glow(K) = \liminf_{N\to\infty}\Esf{K}{N}
\qquad\text{and}\qquad\gup(K) = \limsup_{N\to\infty} \Esf{K}{N}.$$ If
$\glow(K) = \gup(K)$ we denote the common value by $\gex(K)$. With
this we prove the following theorem.
\begin{theorem}\label{th:1}
Let $A\subset \Rp$ be the disjoint union of two compact sets $A_1$ and
$A_2$ of Hausdorff dimension $d>0$ satisfying the following for some
$s>d$:
\begin{itemize}
\item[1.] $\diam(A_1)$ and $\diam(A_2)$ are both less than
  $d(A_1,A_2) := \inf \{ |x-y| : x\in A_1,\,y\in A_2\}$, 
\item[2.] $0<\glow(A_1)<\gup(A_2)<\infty$,
\item[3.] $\gex(A_2)$ exists and is positive and finite.
\end{itemize}
Then the sequence of measures $\{\mu^{s,N}\}_{N=2}^\infty\subset \MAp$ cannot
converge in the weak-star topology on $\MAp$.
\end{theorem}

The central idea behind the proof of Theorem~\ref{th:1} is that, if
$\Es(A_1, N)$ oscillates to leading order, i.e. $N^{1+s/d}$, then the
ratio of the number of points in $\omega_N^s \cap A_1$ to the number
of points in $\omega_N^s \cap A_2$ cannot be constant.  This is
sufficient to show that if one chooses a Urysohn function $\phi$ that
is $1$ on $A_1$ and $0$ on $A_2$, then the limit
$$
\lim_{N\to\infty} \int \phi d\mu^{s,N}
$$
cannot exist, which is enough to show that
$\{\mu^{s,N}\}_{N=2}^\infty$ cannot have a weak-star limit.

The rest of this paper is organized as follows: Section~\ref{sec:wscp}
describes the class of sets that we consider and provides an example
in this class based on the results in~\cite{BHS2}. A particular
weak-star cluster point as $N$ tends to infinity for the discrete
minimal $N$-point energy is demonstrated.  Section~\ref{sec:roc}
proves a lemma regarding the rate of change of the ratio of the
minimal $N$-point energy divided by its leading order in $N$.
Section~\ref{sec:nc} uses the results from the previous sections to
show that, for the ranges of $s$ and sets under consideration, there
is no weak-star limit to the sequence $\{\mu^{s,N}\}_{N=2}^\infty$.

\section{The Set $A$ and a Weak Star Cluster Point}\label{sec:wscp}

We shall consider a set $A$ that is the disjoint union of two compact
set $A_1$ and $A_2$.  We require that $\diam(A_1)$ and $\diam(A_2)$
are both less than $d(A_1,A_2) = \inf\{|x-y|:x\in A_1, y\in A_2\}$. We
shall further restrict ourselves to such sets satisfying
$0<\glow(A_1)<\gup(A_1)<\infty$ and $0<\gex(A_2)<\infty$.

\subsection{An example of such a set}

In~\cite{HardinSaff1} Hardin and Saff show that if $A_2$ is a compact
subset of a $d$-dimensional $C^1$-manifold embedded in $\Rp$,
satisfying $0<\Hd(A_2)<\infty$, then $0<\glow(A_2)=\gup(A_2)<\infty$.
In Proposition 2.6 of~\cite{BHS2} Borodachov, Hardin and Saff show
that if $A_1$ belongs to a certain class of self-similar
$d$-dimensional fractals, then $0<\glow(A_1)<\gup(A_1)<\infty$.  The
set $A_1$ belongs to this class if
\begin{equation}\label{eq:1}
A_1=\bigcup_{i=1}^K \varphi_i(A_1),
\end{equation}
where each $\varphi_i:\Rp\to\Rp$ is a similitude with scaling $L$ and where
$\varphi_i(A_1)\cap\varphi_j(A_1)=\varnothing$ for all $i\ne j$.  It
is a consequence of results by Moran presented in~\cite{Moran1} that
there is a unique compact set $A_1$ that satisfies the condition in
Equation~\eqref{eq:1}, that the Hausdorff dimension of $A_1$ denoted
by $d$ is the solution of the equation $1=KL^d$ and that
$0<\Hd(A_1)<\infty$.  Cantor sets are a prominent example from this
class. 

For a concrete example consider the set that is the union of $$A_2
=\{(x_1,x_2)\in\RealVect{2} : x_1\in[3,4] \text{ and } x_2 = 0\}$$ and
the set $A_1$ satisfying
$$
A_1 = \bigcup_{i=1}^4 \varphi_i (A_1),
$$
where 
$$
\varphi_1(x) = \frac{x}{4},\qquad
\varphi_2(x) = \frac{x}{4} + \left(\frac{3}{4}, 0\right),\qquad
\varphi_3(x) = \frac{x}{4} + \left(0 , \frac{3}{4}\right)\qquad
\text{and}\qquad
\varphi_4(x) = \frac{x}{4} + \left(\frac{3}{4}, \frac{3}{4}\right).
$$
The dimension of both $A_1$ and $A_2$ is $1$, and both have positive
and finite ${\mathcal H}^1$ measure. 

\subsection{A weak-star cluster point of $\{\mu^{s,N}\}_{N=2}^\infty\subset\MAp$}

Here we shall identify a cluster point in the weak-star topology on
$\MAp$ of the sequence of measures $\{\mu^{s,N}\}_{N=2}^\infty$. This
cluster point is one in which the $s$-energy of the sequence of
configurations $\{\omega_N^s\cap A_1\}_{N=2}^\infty$ achieves
$\glow(A_1)$. We proceed by drawing on ideas and techniques developed
by Hardin and Saff in~\cite{HardinSaff1}.

An upper bound for $\displaystyle{\frac{\Es(A,N)}{N^{1+s/d}}}$ can be
obtained as follows: Choose natural numbers $M_1$ and $M_2 = N - M_1$
and consider a configuration of points $\tilde \omega_N$ such that
$\#\tilde \omega_N \cap A_1 = M_1$ and $\#\tilde \omega_N \cap A_2 =
M_2$, and where $E_s(\tilde \omega_N\cap A_1) = \Es(A_1, M_1)$ and
$E_s(\tilde \omega_N\cap A_2) = \Es(A_2, M_2)$. Here the $\#$ denotes
the number of elements in the set following it.  Bounding from above
the interaction energy of points on $A_1$ and points on $A_2$ by
$d(A_1, A_2)^{-s}N^2$ and dividing by $N^{1+s/d}$ gives
\begin{equation}\label{eq:2}
\frac{\Es(A,N)}{N^{1+s/d}} \le
\frac{E_s(\tilde \omega_N)}{N^{1+s/d}} \le
\left(\frac{M_1}{N}\right)^{1+s/d}\Esf{A_1}{M_1} +
\left(\frac{M_2}{N}\right)^{1+s/d}\Esf{A_2}{M_2} +
d(A_1,A_2)^{-s}\frac{N^2}{N^{1+s/d}}.
\end{equation}

Let $\left\{\underline
N_n^1\right\}_{n=1}^\infty\subset\mathbb{N}$ be an increasing
sequence so that 
$$
\lim_{n\to\infty} \Esf{A_1}{\underline N_n^1} = \glow(A_1).
$$
For $\alpha \in (0,1)$ let $\{N_n^\alpha\}_{n=1}^\infty \subset
\mathbb{N}$ be the sequence defined by $N_n^\alpha =
\left\lfloor\frac{1}{\alpha}\underline N_n^1\right\rfloor$. Applying
the bound given in Inequality~\eqref{eq:2} where $N$ is chosen to be
$N_n^\alpha$, $M_1$ is chosen to be $\underline N_n^1$ and $M_2$ is
chosen to be $N_n^\alpha - \underline N_n^1$ gives
$$
\Esf{A}{N_n^\alpha} \le
\left(\frac{\underline
  N_n^1}{N_n^\alpha}\right)^{1+s/d}\Esf{A_1}{\underline N_n^1} + 
\left(\frac{N_n^\alpha - \underline
  N_n^1}{N_n^\alpha}\right)^{1+s/d}\Esfp{A_2}{N_n^\alpha - \underline N_n^1} + 
d(A_1,A_2){N_n^\alpha}^{1-s/d}
$$
For every $\varepsilon > 0$ we may find an $N_0 = N_0(\varepsilon)$
sufficiently high so that for every $N_n^\alpha > N_0$ 
\begin{equation}\label{eq:3}
\Esf{A}{N_n^\alpha} \le
\alpha^{1+s/d}\glow(A_1) + (1-\alpha)^{1+s/d}\gex(A_2) + \varepsilon.
\end{equation}
The unique value of $\alpha$ that minimizes the right hand side of
Inequality~\eqref{eq:3} is 
$$
\alpha^* = \frac{\gex(A_2)^{d/s}}{\glow(A_1)^{d/s} + \gex(A_2)^{d/s}}.
$$
Define $\{N_n^{\alpha^*}\}_{n=1}^\infty$ by $N_n^{\alpha^*} =
\left\lfloor \frac{1}{\alpha^*} \underline N_n^1\right\rfloor$.

To obtain a lower bound for $\Esf{A}{N}$ begin by defining the
function $N_1:\mathbb{N}\to\mathbb{N}$ by
$$
N_1(N) = \min_{\omega_N^s \subset A}\#(\omega_N^s\cap A_1).
$$ Because an $N$-point $s$-energy-minimizing configuration $\omega_N^s$
may not be unique, we take a minimum over all $N$-point
$s$-energy-minimizing configurations to determine the value $N_1(N)$.  We
define $N_2(N)$ to be $N - N_1(N)$. Note that $\displaystyle{N_2(N)
  \ge \min_{\omega_N^s \subset A}\#(\omega_N^s\cap A_2)}$.

Given $N$ and a minimal $N$-point configuration $\omega_N^s$ such that
$N_1(N) = \#\omega_N^s\cap A_1$, we may bound from below $\Es(A,N) =
E_s(\omega_N^s)$ by discarding the interaction energy between points
in $\omega_N^s \cap A_1$ and points in $\omega_N^s \cap A_2$ and then
further replacing the points in $\omega_N^s\cap A_1$ with a minimal
$N_1(N)$-point configuration in $A_1$ and replacing the points in
$\omega_N^s\cap A_2$ with a minimal $N_2(N)$-point configuration in
$A_2$. The bound is then
\begin{equation}\label{eq:4}
\Esf{A}{N} \ge 
\left(\frac{N_1(N)}{N}\right)^{1+s/d}\Esf{A_1}{N_1(N)} +
\left(\frac{N_2(N)}{N}\right)^{1+s/d}\Esf{A_2}{N_2(N)}. 
\end{equation}
We would like to improve the bound above by employing the asymptotic
properties of $\displaystyle{\Esf{A_1}{N_1(N)}}$ and
$\displaystyle{\Esf{A_1}{N_1(N)}}$.  This will require ensuring that
both $N_1(N)$ and $N_2(N)$ tend to infinity as $N$ tends to infinity.
This will be accomplished by applications of Lemma~\ref{lm:1}, which
employs ideas presented by Bj\"orck in~\cite{Bjorck1}.

\begin{lemma}\label{lm:1}
Let $B_1$, $B_2 \subset\Rp$ be compact and of dimension $d>0$. Further
suppose $\diam(B_2) < d(B_1, B_2)$. Let $s>d$ and define 
$$
\tilde N_1(N) = \min_{\omega_N^s\subset B_1\cup B_2} \# (\omega_N^s\cap B_1),
$$
where the minimum is taken over all $N$-point minimal $s$-energy
configurations within $B_1\cup B_2$. Then 
$$
\liminf_{N\to\infty} \tilde N_1(N) = \infty.
$$
\end{lemma}
\begin{proof}
For sake of contradiction assume that 
$$
\liminf_{N\to\infty} \tilde N_1(N) = L < \infty,
$$
then there is an increasing sequence $\{N_i\}_{i=1}^\infty \subset
\mathbb{N}$ so that $\tilde N_1(N_i) = L$ for all $i\in \mathbb{N}$.
The contradiction shall be that, for large enough $N$, a configuration
that only places $L$ points on $B_1$ cannot have minimal energy.

We first show that the quantity 
$$
R = \inf_{\begin{array}{c}K\subset B_1 \\ \# K = L\end{array}}
\max_{y\in B_1} d (y, K)
$$ is positive. We may choose a sequence of $L$-point configurations
in $B_1$, $\{K_n\}_{n=1}^\infty$, so that
$$
\lim_{n\to\infty} \max_{y\in B_1} d (y, K_n) = R,
$$ and because the $L$-fold product of $B_1$ with itself is compact,
we may choose $\{K_n\}_{n=1}^\infty$ to be convergent to some $K^*$.
Continuity of the function $B_1^L \ni K\to \max_{y\in B_1} d (y, K)\in\Rl$ allows us
to conclude 
$$
\max_{y\in B_1} d (y, K^*) = R.
$$
The set $B_1$ is of positive Hausdorff dimension so it is infinite,
which is sufficient to conclude $R>0$.

If $\tilde N_1(N_i) = L$ for all $i$, then for every $i$ there is a
point $r_i \in B_1$ that is separated from $\omega_{N_i}^s$ by at
least $R$ for any minimal $N_i$-point configuration $\omega_{N_i}^s
\subset B_1\cup B_2$.  We may bound the potential energy at $r_i$ due
to the point in $\omega_{N_i}^s$ from above by
$$
U(r_i) = \sum_{x\in\omega_{N_i}^s}\Rk{r_j}{x}{s}
\le LR^{-s} + (N_i -L)d(B_1, B_2)^{-s},
$$
where the first term is an upper bound for the contribution of the $L$
points of $\omega_{N_i}^s \cap B_1$ and the second term is an upper
bound for the $N_i - L$ points of $\omega_{N_i}^s \cap B_2$.

Alternatively, given any point $x_j\in \omega_{N_i}^s \cap B_2$ we may
bound from below its point energy due to the other points by 
$$
U_j(x_j) = \sum_{x\in \omega_{N_i}^s \backslash \{x_j\}} 
\Rk{x_j}{x}{s} \ge (N_i - L - 1)\diam (B_2)^{-s}.
$$ In this lower bound we have excluded the contribution to the point
energy at $x_j$ from the points in $\omega_{N_i}^s \cap B_1$, and
bounded from below the contribution to the point energy at $x_j$ from any 
point in $\omega_{N_i}^s \cap B_2\backslash \{x_j\} $ by
$\diam(B_2)^{-s}$.  

Because $d(B_1, B_2) > \diam(B_2)$, we may find $N_i$ sufficiently
high so that $U(r_i) < U_j(x_j)$, but then the energy of configuration
$\omega_{N_i}^s$ could be reduced by moving the $j^\text{th}$ point
from $x_j$ to $r_j$ and this contradicts the assumption that
$\omega_{N_i}^s$ has minimal energy.
\end{proof}

We may apply Lemma~\ref{lm:1} to the sets under consideration by
identifying either $A_1$ or $A_2$ as $B_1$. This is sufficient to show
that
$$
\liminf_{N\to\infty}N_1(N) = \liminf_{N\to\infty}N_2(N) = \infty.
$$

Along the the sequence $\{N_n^{\alpha^*}\}_{n=1}^\infty$, we may apply
the bound given in Inequality~\eqref{eq:3} and, because $N_1(N_n^{\alpha^*})$ and
$N_1(N_n^{\alpha^*})$ grow to infinity, we may apply
Inequality~\eqref{eq:4} giving, for $N_n^{\alpha^*}$ sufficiently high
$$
{\alpha^*}^{1+s/d}\glow(A_1) + (1-\alpha^*)^{1+s/d}\gex(A_2) + 2\varepsilon
\ge
\left(\frac{N_1(N_n^{\alpha^*})}{N_n^{\alpha^*}}\right)^{1+s/d}\glow(A_1) +
\left(\frac{N_2(N_n^{\alpha^*})}{N_n^{\alpha^*}}\right)^{1+s/d}\gex(A_2).
$$ Here we have bounded from below
$\displaystyle{\Esf{A_1}{N_1(N_n^{\alpha^*})}}$ by $\glow(A_1) -
\varepsilon/2$ and bounded from below
$\displaystyle{\Esf{A_2}{N_2(N_n^{\alpha^*})}}$ by $\gex(A_2) -
\varepsilon/2$.

Because $\varepsilon$ is arbitrary and because $\alpha^*$ is the
unique minimizer of the left hand side of the above upper bound we
may conclude that 
$$
\lim_{n\to\infty}\frac{N_1(N_n^{\alpha^*})}{N_n^{\alpha^*}} = \alpha^*.
$$

If we let $\psi\in\MAp$ be any weak-star cluster point of
$\{\mu^{s,N_n^{\alpha^*}}\}_{n=1}^\infty$, and if we choose a
continuous function $\phi:A\to\Rl$ such that $\phi(x) = 1$ for all
$x\in A_1$ and
$\phi(y) = 0$ for all $y \in A_2$, then 
$$
\int\phi d\psi = \psi(A_1) = \alpha^*.
$$
This is sufficient to prove Lemma~\ref{lm:2}
\begin{lemma}\label{lm:2}
Let $A$ be the disjoint union of two compact sets $A_1$ and $A_2$ that
meet the following conditions:
\begin{itemize}
\item [1.] both $A_1$ and $A_2$ are of Hausdorff dimension $d$,
\item [2.] $\diam(A_1)$ and $\diam(A_2)$ are both less than $d(A_1,
  A_2)$, 
\item [3.] for some $s>d$, $0<\glow(A_1)\le \gup(A_1) < \infty$ and 
\item [4.] for the same $s$, $\gex(A_2)$ exists and $0<\gex(A_2) < \infty$.
\end{itemize}
Then there is a weak-star cluster-point $\psi\in\MAp$ of the sequence
of measures $\{\mu^{s,N}\}_{N=2}^\infty$ so that 
$$
\psi(A_1) = \frac{\gex(A_2)}{\glow(A_1) + \gex(A_2)}.
$$
\end{lemma}

\section{Rate of Change in $\displaystyle{\Esf{K}{N}}$ for $K\subset\Rp$ Compact}\label{sec:roc}

In this section we consider a compact set $K\subset\Rp$ so that
$0<\gup(K)<\infty$, and bound from below the rate of change of the
function $G:\mathbb{N} \to \mathbb{R}$ defined by
$$
G(N) = \Esf{K}{N}.
$$ This is the content of Lemma~\ref{lm:3} In Corollary~\ref{cor:1} we
use Lemma~\ref{lm:3} to show that, if for some $N_0$, $G(N_0)$ is
close to $\gup(K)$, then for some $N'$ that is larger than $N_0$ and
for which the ratio $N_0/N'$ is close to $1$, we will have that
$G(N')$ is also close to $\gup(K)$.

\begin{lemma}\label{lm:3}
Let $K\subset\Rp$ be compact and $s > d = \dim K >0$.  Let $N \ge 2$
and $N' \ge 1$ be natural numbers and let $\kappa = \frac{N'}{N}$.
Then
$$
G((1+\kappa)N) \ge G(N) - \left(1+\frac s d \right)\kappa G(N).
$$
\end{lemma}
\begin{proof}
From our assumptions
\begin{eqnarray*}
G((1+\kappa)N) = \Esfp{K}{N+N'} >&
\left(\frac{N}{N+N'}\right)^{1+s/d}\Esf{K}{N} \\
=& \left( \frac{1}{1 + \kappa}\right)^{1+s/d}G(N) \\
\ge & \left (1 - \left(1 + \frac s d\right)\kappa \right) G(N) \\ 
=& G(N) - \left(1+\frac s d \right)\kappa G(N),
\end{eqnarray*}
where the last inequality follows from a first-order expansion of
$h(\kappa) := \left( \frac{1}{1 + \kappa}\right)^{1+s/d}$ about zero
and the convexity of $h$.
\end{proof}

\begin{corollary}\label{cor:1}
Let $K\subset\Rp$ be compact and $s$, $d>0$ so that $0<\gup(K)<\infty$. Let
$\{M_n\}_{n=1}^\infty\subset\mathbb{N}$ be an increasing sequence so
that 
$$
\lim_{n\to\infty}\Esf{K}{M_n} = \gup(K).
$$
For every $\varepsilon >0$ there is a $\delta >0$ and an $M_0 < \infty$ so that
if, for some $M'\in\mathbb{N}$ and $\tilde
M\in\{M_n\}_{n=1}^\infty$,
$$M_0<\tilde M< M' < \frac{\tilde M}{1-\delta},$$
then 
$$
\Esf{K}{M'}\ge \gup(K) - \varepsilon.
$$
\end{corollary}
\begin{proof}
Let $\varepsilon >0$ be arbitrary. Choose $M_0$ so that for any $\tilde
M \in \{M_n\}_{n=1}^\infty$ greater than $M_0$
\begin{equation}\label{eq:6}
\Esf{K}{\tilde M} > \gup(K) - \frac{\varepsilon}{2}
\end{equation}
Choose $\kappa_0$ sufficiently small so that for all $\kappa \in (0,
\kappa_0)$, 
\begin{equation}\label{eq:7}
(1 + s/d)\kappa \sup_{M>M_0}\Esf{K}{M} < \frac{\varepsilon}{2}.
\end{equation}
Combining Lemma~\ref{lm:3} and Inequalities~\eqref{eq:6} and~\eqref{eq:7}
gives for any $\tilde M \in \mathbb{N}$ and $\kappa \in (0,\kappa_0)$
where $(1+\kappa)\tilde M \in \mathbb{N}$
\begin{equation}\label{eq:8}
\Esfp{K}{(1+\kappa)\tilde M} \ge \Esf{K}{\tilde M} -
(1+s/d)\kappa\Esf{K}{\tilde M} \ge
\gup(K) - \varepsilon.
\end{equation}
Choose $\displaystyle{\delta = \frac{\kappa_0}{1+\kappa_0}}$. If $M'$
and $\tilde M$ are such that $M_0< \tilde M < M' < \frac{\tilde
M}{1-\delta}$, and if $M' = (1+\kappa)\tilde M$, then $\kappa <
\kappa_0$ and the bound in Inequality~\eqref{eq:8} ensures
$$
\Esf{K}{M'} \ge \gup(K) - \varepsilon.
$$
\end{proof}
\section{Non-Convergence of
  $\displaystyle{\{\mu^{s,N}\}_{N=2}^\infty}$ in the Weak-Star Topology}\label{sec:nc}

In this section we prove Theorem~\ref{th:1} -- that the sequence of
measures $\{\mu^{s,N}\}_{N=2}^\infty\subset\MAp$ cannot converge in
the weak-star topology on $\MAp$. If the sequence did converge, it
must converge to the measure $\psi$ identified in Lemma~\ref{lm:2},
this will be shown to lead to a contradiction.

\begin{proof}[Proof of Theorem~\ref{th:1}]
For sake of contradiction, assume that the sequence
$\{\mu^{s,N}\}_{N=2}^\infty\subset \MAp$ converged in the weak-star
topology on $\MAp$.  Since this topology separates elements of $\MAp$
and since, by Lemma~\ref{lm:1}, a subsequence of
$\{\mu^{s,N}\}_{N=2}^\infty$ converges to $\psi$, the assumption would
require that $\mu^{s,N}$ converges to $\psi$ in the weak-star sense as
$N\to\infty$ implying
\begin{equation}\label{eq:9}
\lim_{N\to\infty}\frac{N_1(N)}{N} = \alpha^* = 
\frac{\gex(A_2)}{\glow(A_1) + \gex(A_2)}
\end{equation}

Let $\left\{\overline N_n^1\right\}_{n=1}^\infty \subset \mathbb{N}$ be an
increasing sequence so that 
$$
\lim_{n\to\infty} \Esf{A_1}{\overline N_n^1} = \gup (A_1).
$$

Our intention is to find $N'''\in \left\{\overline
N_n^1\right\}_{n=1}^\infty$ and an $N' \in \mathbb{N}$ so that we may
apply Corollary~\ref{cor:1} where $\tilde M$ is identified with $N'''$ and
$M'$ is identified with $N_1(N')$.

We proceed as follows: Let $\varepsilon > 0$. Let $\delta$ and $M_0$
be as provided by Corollary~\ref{cor:1} applied to the case $K = A_1$ and
$\{M_n\}_{n=1}^\infty = \left\{\overline N_n^1\right\}_{n=1}^\infty$.
One may verify that it is possible to choose $\gamma\in (0,1)$ such that, for
$\displaystyle{\gamma' = \frac{2\gamma}{1-\gamma}}$, the following
(motivated by inequalities arising later in the proof)
hold:
$$
\gamma' + (1+\gamma')\left(1 - \frac{1}{1+3\gamma}\right) < \delta
\qquad\text{and}\qquad 
\frac{1 + \frac{5}{2}\gamma}{1 + \gamma'} \ge 1.
$$ We choose such a $\gamma$. Define the non-decreasing sequence
$\{\tilde N_n\}_{n=1}^\infty$ by the equation $\tilde N_n =
\left\lfloor(1+3\gamma)\overline N_n^1\right\rfloor$.  

From Equation~\eqref{eq:9} and from Lemma~\ref{lm:2}, we may further
increase $M_0$ so that the following all hold
\begin{equation}\label{eq:10}
\sup_{N\ge M_0}\left|\frac{N_1(N)}{N} - \alpha^*\right| < 
\min \{\gamma \alpha^*, \varepsilon\},
\end{equation}
\begin{equation}\label{eq:11}
\sup_{N\ge M_0}\Esf{A_1}{N_1(N)} \le \gup(A_1) + \varepsilon \qquad\text{and}
\end{equation}
\begin{equation}\label{eq:12}
\sup_{N\ge M_0}\left|\Esf{A_2}{N_2(N)} - \gex(A_2)\right| < \varepsilon.
\end{equation}
$$
\left(1 + \frac{5}{2}\gamma\right)N <
\left\lfloor(1+3\gamma)N\right\rfloor \qquad \text{for all $N>M_0$}.
$$
Inequality~\eqref{eq:10} implies
\begin{equation}\label{eq:13}
\sup_{N\ge M_0}\left|\frac{N_1(N)}{\alpha^*N} - 1\right| < \gamma.
\end{equation}

Choose $N'\in \mathbb{N}$ and $N''\in \{\tilde N_n\}_{n=1}^\infty$ so
that both are greater than $M_0$, and so that 
\begin{equation}\label{eq:14}
\left|1 - \frac{N''}{\alpha^*N'}\right| < \gamma.
\end{equation}
This allows the following bound
\begin{eqnarray}
\label{eq:15}
\left|1-\frac{N''}{N_1(N')}\right| =&
\left|1 - \frac{N''}{\alpha^*N'} + 
\frac{N''}{\alpha^*N'} - \frac{N''}{N_1(N')}\right| \\ 
\nonumber
\le & 
\left|1 - \frac{N''}{\alpha^*N'}\right| + 
\left|\frac{N''}{\alpha^*N'} - \frac{N''}{N_1(N')}\right| \\ 
\nonumber
\le &
\gamma + \left(\frac{N''}{\alpha^*N'}\right)
\left|1 - \frac{\alpha^*N'}{N_1(N')}\right|\\ 
\nonumber
\le & 
\gamma + (1 + \gamma)\frac{\gamma}{1-\gamma} \\ 
\nonumber
=& \frac{2\gamma}{1-\gamma} = \gamma'.
\end{eqnarray}
The last inequality above follows from Inequalities~\eqref{eq:13}
and~\eqref{eq:14}.  Now let $N'''$ denote the element of $\left\{\overline
N_n^1\right\}_{n=1}^\infty$ so that $N'' = \left\lfloor
(1+3\gamma)N'''\right\rfloor$. Then
\begin{eqnarray*}
\left| 1 - \frac{N'''}{N_1(N')}\right| =&
\left| 1 - \frac{N''}{N_1(N')} + \frac{N''}{N_1(N')}
-\frac{N'''}{N_1(N')} \right| \\
\le &
\left| 1 - \frac{N''}{N_1(N')}\right| + 
\left|\frac{N''}{N_1(N')}-\frac{N'''}{N_1(N')} \right| \\
\le & 
\gamma' + \frac{N''}{N_1(N')}\left|1 -\frac{N'''}{N''} \right|\\
\le & 
\gamma' + (1+\gamma')\left|1 -\frac{N'''}{\left\lfloor
  1+3\gamma\right\rfloor N'''} \right| \\
\le & 
\gamma' + (1+\gamma')\left(1 -\frac{1}{1+3\gamma } \right).\\
\end{eqnarray*}
Here the second to last inequality follows from
Inequality~\eqref{eq:15}. Inequality~\eqref{eq:15} also implies 
$$
1-\frac{N''}{N_1(N')} \ge -\gamma',\quad\text{hence}\quad
N_1(N') \ge \frac{N''}{\gamma' + 1} 
= \frac{\left\lfloor (1+3\gamma)N'''\right\rfloor}{\gamma' + 1}
\ge N'''\frac{ \left(1+\frac{5}{2}\gamma\right) }{\gamma' + 1}
$$
Because $\gamma$ was chosen to ensure that
$$
\frac{ \left(1+\frac{5}{2}\gamma\right) }{\gamma' + 1} \ge 1,
\qquad\text{and}\qquad
\gamma' + (1+\gamma')\left(1 - \frac{1}{1+3\gamma}\right) < \delta
$$
we have that $N_1(N') \ge N'''$ and 
$$
1 - \frac{N'''}{N_1(N')} < \delta.
$$ 
With this we may employ Corollary~\ref{cor:1} where $\tilde M$ is
identified with $N'''$ and $M'$ is identified with $N_1(N')$, giving the
bound 
$$
\Esf{A_1}{N_1(N')} \ge \gup(A_1) - \varepsilon.
$$
This, combined with Inequalities~\eqref{eq:4} and~\eqref{eq:12}, gives
$$
\Esf{A}{N'} \ge \left(\frac{N_1(N')}{N'}\right)^{1+s/d}
(\gup(A_1) - \varepsilon) + 
\left(N' - \frac{N_1(N')}{N'}\right)^{1+s/d}
(\gex(A_2) - \varepsilon).
$$
In light of Inequality~\eqref{eq:10} we have
$$
\Esf{A}{N'} \ge (\alpha^* - \varepsilon))^{1+s/d}
(\gup(A_1) - \varepsilon) + 
(1 - \alpha^* - \varepsilon)^{1+s/d}
(\gex(A_2) - \varepsilon).
$$

In the preceding arguments, $\varepsilon$ was an arbitrary positive
number that constrained $N'$ to be above some value.  We may then
construct a sequence $\{N_n'\}_{n=1}^\infty$ of such $N'$ so that 
$$
\lim_{n\to\infty} \Esf{A}{N_n'} \ge {\alpha^*}^{1+s/d}\gup(A_1) + (1 -
\alpha^*)^{1+s/d}\gex(A_2). 
$$

Now, for each $N_n' \in \{N_n'\}_{n=1}^\infty$ that is larger than
$2M_0$, and for any two natural
numbers $M_n^1$ and $M_n^2 = N_n' - M_n^1$ both greater than $M_0$,
Inequalities~\eqref{eq:2},~\eqref{eq:11} and~\eqref{eq:12} give
$$
\Esf{A}{N_n'} \le 
\left(\frac{M_n^1}{N_n'}\right)^{1+s/d}(\gup(A_1) + \varepsilon) +
\left(\frac{M_n^2}{N_n'}\right)^{1+s/d}(\gex(A_2) + \varepsilon) +
{N_n'}^{1-s/d}d(A_1,A_2).
$$ For an arbitrary $\beta \in (0,1)$ we may choose $M_n^1$ and
$M_n^2$ so that $\frac{M_n^1}{N_n'} \to \beta$ and $\frac{M_n^2}{N_n'}
\to 1 -\beta$ as $n\to\infty$. Note also that, because $N_n'$ goes to
infinity as $n$ goes to infinity, the term ${N_n'}^{1-s/d}d(A_1,A_2)$
goes to zero as $n$ grows. For any $\beta$ we obtain the following
bound
\begin{equation}\label{eq:16}
{\beta}^{1+s/d}\gup(A_1) + (1 - \beta)^{1+s/d}\gex(A_2) \ge
\lim_{n\to\infty} \Esf{A}{N_n'} \ge 
{\alpha^*}^{1+s/d}\gup(A_1) + (1 - \alpha^*)^{1+s/d}\gex(A_2). 
\end{equation}
The function $f(\beta) = {\beta}^{1+s/d}\gup(A_1) + (1 -
\beta)^{1+s/d}\gex(A_2)$ is uniquely minimized by 
$$
\beta = 
\frac{\gex(A_2)}{\gup(A_1) + \gex(A_2)} 
\ne 
\frac{\gex(A_2)}{\glow(A_1) + \gex(A_2)} 
= \alpha^*,
$$
invalidating Inequality~\eqref{eq:16}. This is sufficient to show that
$\{\mu^{s,N}\}_{N=2}^\infty$ cannot converge in the weak-star topology
on $\MAp$.
\end{proof}

\bibliography{References}
\bibliographystyle{abbrv}

\end{document}